\documentclass[english]{article}
\usepackage[T1]{fontenc}
\usepackage[latin9]{inputenc}
\usepackage{babel}
\usepackage{mathrsfs}
\usepackage{amsmath}
\usepackage{amsthm}
\usepackage{amssymb}
\usepackage[all]{xy}
\usepackage[unicode=true,pdfusetitle,
 bookmarks=true,bookmarksnumbered=false,bookmarksopen=false,
 breaklinks=false,pdfborder={0 0 1},backref=false,colorlinks=false]
 {hyperref}

\makeatletter
\theoremstyle{plain}
\newtheorem{thm}{\protect\theoremname}
\theoremstyle{definition}
\newtheorem{example}[thm]{\protect\examplename}
\theoremstyle{plain}
\newtheorem{lem}[thm]{\protect\lemmaname}
\theoremstyle{plain}
\newtheorem{prop}[thm]{\protect\propositionname}
\theoremstyle{plain}
\newtheorem{cor}[thm]{\protect\corollaryname}

\usepackage{babel}

\newcommand{\xyR}[1]{%
\makeatletter
\xydef@\xymatrixrowsep@{#1}
\makeatother
} 

\newcommand{\xyC}[1]{%
\makeatletter
\xydef@\xymatrixcolsep@{#1}
\makeatother
} 

\usepackage{xcolor}

\DeclareMathOperator*{\medwedge}{{\textstyle\bigwedge}}

\providecommand{\examplename}{Example}
\providecommand{\lemmaname}{Lemma}
\providecommand{\propositionname}{Proposition}
\providecommand{\theoremname}{Theorem}

\providecommand{\corollaryname}{Corollary}

\makeatother

\providecommand{\corollaryname}{Corollary}
\providecommand{\examplename}{Example}
\providecommand{\lemmaname}{Lemma}
\providecommand{\propositionname}{Proposition}
\providecommand{\theoremname}{Theorem}

\begin{document}
\title{Free-lattice functors weakly preserve epi-pullbacks}
\author{H.\,Peter Gumm, Ralph Freese}
\maketitle
\begin{abstract}
Suppose $p(x,y,z)$ and $q(x,y,z)$ are terms. If there is a common
``ancestor'' term $s(z_{1},z_{2},z_{3},z_{4})$ specializing to
$p$ and $q$ through identifying some variables 
\begin{align*}
p(x,y,z) & \approx s(x,y,z,z)\\
q(x,y,z) & \approx s(x,x,y,z),
\end{align*}
then the equation 
\[
p(x,x,z)\approx q(x,z,z)
\]
is trivially obtained by syntactic unification of $s(x,y,z,z)$ with
$s(x,x,y,z).$

In this note we show that for lattice terms, and more generally for
terms of lattice-ordered algebras, the converse is true, too. Given
terms $p,q,$ and an equation 
\begin{equation}
p(u_{1},\ldots,u_{m})\approx q(v_{1},\ldots,v_{n})\label{eq:p_eq_q}
\end{equation}
where $\{u_{1},\ldots,u_{m}\}=\{v_{1},\ldots,v_{n}\},$ there is always
an ``ancestor term'' $s(z_{1},\ldots,z_{r})$ such that $p(x_{1},\ldots,x_{m})$
and $q(y_{1},\ldots,y_{n})$ arise as substitution instances of $s,$
whose unification results in the original equation (\ref{eq:p_eq_q}). 
\end{abstract}
\begin{quote}
In category theoretic terms the above proposition, when restricted
to lattices, has a much more concise formulation:
\end{quote}
\begin{quotation}
\emph{Free-lattice functors weakly preserves pullbacks} \emph{of epis}.
\end{quotation}

\section{Introduction}

The motivation of this study arose from coalgebra. When studying $F$-coalgebras
for a type functor $F$, the limit preservation properties of $F$
are known to exert a crucial influence on the structure theory of
the class of all $F$-coalgebras. In particular, \emph{weak preservation
of pullbacks} is a property which many familiar $Set$-endofunctors
enjoy. However, there are notable exceptions, such as the bounded
finite powerset functor or the neighborhood functor. Therefore, weak
preservation of pullbacks was split into two easier but separate conditions,
namely \emph{weak preservation of kernel pairs} and (weak)\emph{ preservation
of preimages}. It could be shown that these two simpler conditions
combine to be equivalent to weak preservation of pullbacks. Later
it was also found that the first condition, weak preservation of kernel
pairs, is equivalent to weak preservation of pullbacks of epis.

From a universal algebraic point of view it is of interest to study
the functors $F_{\mathcal{V}}$ which, for a given variety of algebras
$\text{\ensuremath{\mathcal{V}}},$ send a set $X$ to the free algebra
$F_{\mathcal{V}}(X)$ and a map $\sigma:X\to Y$ to its homomorphic
extension $\bar{\sigma}:F_{\mathcal{V}}(X)\to F_{\mathcal{V}}(Y).$
Note that for an arbitrary term $p(x_{1},\ldots,x_{n})$ we have that
\begin{equation}
\bar{\sigma}\,p(x_{1},\ldots,x_{n})=p(\sigma x_{1},\ldots,\sigma x_{n}),\label{eq:homomorphic extension}
\end{equation}
where on the left hand side of the equation $p$ is understood to
be evaluated in $F_{\mathcal{V}}(X)$ and on the right hand side in
$F_{\mathcal{V}}(Y).$

In this context, weak preservation of kernel pairs translates into
an interesting algebraic condition, asserting that given an equation
\begin{equation}
p(u_{1},\ldots,u_{m})\approx q(v_{1},\ldots,v_{n}),\label{eq:p=00003D00003D00003Dq}
\end{equation}
where $\{u_{1},\ldots,u_{m}\}=\{v_{1},\ldots,v_{n}\},$ then the terms
$p$ and $q$ are in fact both obtained from a common ``ancestor''
term $s$ by identifying some of its variables, so that the equality
(\ref{eq:p=00003D00003D00003Dq}) trivially follows from this representation.
\begin{example}
\label{exa:pxxy_eq_qxyy}Assume that $F_{\mathcal{V}}$ weakly preserves
kernel pairs, then for any $\mathcal{V}$-equation 
\begin{equation}
p(x,x,y)\approx q(x,y,y)\label{eq:p(xxy)=00003D00003D00003Dq(xyy)}
\end{equation}
there exists a quaternary term $s$ such that 
\begin{align*}
p(x,y,z) & \approx s(x,y,z,z)\\
q(x,y,z) & \approx s(x,x,y,z).
\end{align*}
This representation then trivially entails (\ref{eq:p(xxy)=00003D00003D00003Dq(xyy)}),
since the most general unifier of $s(x,y,z,z)$ with $s(x,x,y,z)$
is $s(x,x,z,z),$ resulting in the original equation
\[
p(x,x,z)\approx s(x,x,z,z)\approx q(x,z,z).
\]
\end{example}

The ``ancestor condition'' in the previous example has been introduced
in \cite{Gumm20b}. Applying it to the description of $n-$permutable
varieties given by Hagemann and Mitschke \cite{Hagemann1973}, we
could show that for an $n-$permutable variety $\mathcal{V},$ the
functor $F_{\mathcal{V}}$ weakly preserves kernel pairs if and only
if $\mathcal{V}$ is congruence permutable, which in turn holds, according
to Mal'cev~\cite{Malcev1954}, if and only if there exists a term
$m(x,y,z)$ such that the equations 
\begin{align}
m(x,y,y) & \approx x\label{eq:malcev1}\\
m(x,x,y) & \approx y\label{eq:malcev2}
\end{align}
are satisfied.

In this note we are going to show that for any variety $\mathcal{L}$
of lattices, or more generally, of lattice ordered universal algebras,
the free algebra functor $F_{\mathcal{L}}$ weakly preserves kernel
pairs. Therefore, any pair of terms $p,q$ which combine to a valid
equation (\ref{eq:p=00003D00003D00003Dq}) are instances of a common
ancestor term $s,$ so that the equation (\ref{eq:p=00003D00003D00003Dq})
trivially results from instantiations of $s$ resulting in a syntactically
identical term.

From the mentioned paper \cite{Gumm20b} it follows that for no congruence
modular variety $\mathcal{V}$ the functor $F_{\mathcal{V}}$ preserves
preimages. Hence the functor $F_{\mathcal{L}}$ studied in this note
will not preserve preimages, which shows that the variable condition
$\{u_{1},\ldots,u_{m}\}=\{v_{1},\ldots,v_{n}\}$ cannot be dropped.

\section{Preliminaries}

Most of the time we shall omit parentheses when applying unary functions
to arguments, and we assume that application associates to the right,
so we write $fx$ for $f(x)$ and $fgx$ for $f(g(x)).$

For a map $f:X\to Y$ we denote the image of $f$ by $f(X)$ or by
$\text{im}\,f$ and its preimage by $f^{-1}(Y).$ The \emph{kernel}
of $f$ is 
\[
\ker f:=\{(x_{1},x_{2})\in X\times X\mid fx_{1}=fx_{2}\}.
\]

\begin{lem}
\label{lem:diagram_lemma}Given a surjective map $f:X\twoheadrightarrow Y$
and an arbitrary map $g:X\to Z$, then there exists a (necessarily
unique) map $h:Y\to Z$ with $h\circ f=g$ if and only if $\ker f\subseteq\ker g.$
\[
\xymatrix{X\ar@{->>}[r]^{f}\ar[dr]_{g} & Y\ar@{-->}[d]^{h}\\
 & Z
}
\]
\end{lem}

Every surjective map $f:X\twoheadrightarrow Y$ is \emph{right invertible},
i.e. has a \emph{right inverse} which we shall denote by $f^{-}$
and which obeys the equation $f\circ f^{-}=id_{Y}.$ This general
statement is equivalent to the axiom of choice, however we shall need
it here only for finite sets $X$ and $Y.$

\section{Weak preservation of pullbacks}

Given set maps $\alpha:X\to Z$ and $\beta:Y\to Z$, the \emph{pullback}
of $\alpha$ and $\beta$ is a triple $(P,\pi_{1},\pi_{2})$ consisting
of an object $P$ together with maps $\pi_{1}:P\to X$ and $\pi_{2}:P\to Y$
such that
\begin{itemize}
\item $\alpha\circ\pi_{1}=\beta\circ\pi_{2},$ and
\item for every \emph{``competitor}'', i.e. for every other object $Q$
with maps $\eta_{1}:Q\to X,$ $\eta_{2}:Q\to Y$ also satisfying $\alpha\circ\eta_{1}=\beta\circ\eta_{2}$,
there is a \emph{unique} map $d:Q\to P$ such that $\eta_{1}=\pi_{1}\circ d$
and $\eta_{2}=\pi_{2}\circ d.$ 
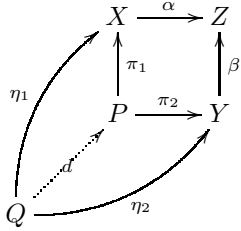
\[
\xymatrix{ & X\ar[r]^{\alpha} & Z\\
 & P\ar[r]^{\pi_{2}}\ar[u]_{\pi_{1}} & Y\ar[u]_{\beta}\\
Q\ar@/_{1pc}/[rru]_{\eta_{2}}\ar@/^{1pc}/[uur]^{\eta_{1}}\ar@{..>}[ur]|d
}
\]
\end{itemize}
By dropping the \emph{uniqueness} requirement, one obtains the definition
of a \emph{weak} \emph{pullback}.

In the category $Set$ of sets and mappings, the pullback of two maps
$\alpha$ and $\beta$ is, up to isomorphism, given by 
\[
Pb(\alpha,\beta)=\{(x,y)\in X\times Y\mid\alpha x=\beta y\}
\]
with $\pi_{1}$ and $\pi_{2}$ being the coordinate projections.
\begin{itemize}
\item If $\alpha=\beta$ then $Pb(\alpha,\beta)$ is just $\ker\alpha$,
the kernel of $\alpha,$ and $(\ker\alpha,\pi_{1},\pi_{2})$ is called
a \emph{kernel pair}.
\item If $\beta$ is injective, then $Pb(\alpha,\beta)\cong\alpha^{-1}(\beta(Y)),$
hence such a pullback is called a \emph{preimage}.
\end{itemize}
Weak pullbacks are always of the shape $(Q,\eta_{1},\eta_{2})$ where
$d:Q\to Pb(\alpha,\beta)$ is right invertible and $\eta_{i}=\pi_{i}\circ d$
for $i=1,2.$

We say that a functor $F$ \emph{weakly preserves pullbacks} if applying
$F$ to a pullback diagram results in a weak pullback diagram. $F$
is said to\emph{ preserve weak pullbacks}, if $F$ transforms weak
pullback diagrams into weak pullback diagrams.

Fortunately, it is easy to see that a functor\emph{ preserves weak
pullbacks} if and only if it \emph{weakly preserves pullbacks}, see
e.g. \cite{Gumm01}, and that a functor \emph{preserves }preimages
if and only if it \emph{weakly preserves} preimages.

For $Set$-endofunctors $F$, weak preservation of pullbacks can be
checked elementwise:
\begin{prop}
\label{prop:pointwise}A $Set$-functor $F$ weakly preserves the
pullback of $\alpha:X\to Z$ and $\beta:Y\to Z$ iff for any $p\in F(X)$
and $q\in F(Y)$ with $r:=(F\alpha)p=(F\beta)q$ there exists some
$s\in F(Pb(\alpha,\beta))$ such that $(F\pi_{1})s=p$ and $(F\pi_{2})s=q.$
\[
\xymatrix{\ar@{}[r]|(0.6){p\,\,\in} & F(X)\ar@{->}[r]^{F\alpha} & F(Z) & \ar@{}[l]|{\ni\,\,r}\\
\ar@{}[r]|(0.4){s\,\,\in} & F(Pb(\alpha,\beta))\ar[r]\sb(0.6){F\pi_{2}}\ar[u]\sp(0.5){F\pi_{1}} & F(Y)\ar[u]_{F\beta} & \ar@{}[l]|{\ni\,\,q}
}
\]
\end{prop}

Since we are only concerned with finitary operations, free-algebra
functors $F_{\mathcal{V}}$ happen to be \emph{finitary}. This means
that given a set $X$ and $p\in F(X)$, there is some finite subset
$X_{0}\subseteq X$ such that $p\in F(X_{0}).$ The following easy
lemma allows us to restrict our consideration to finite sets and maps
between them:
\begin{lem}
If $F$ is finitary, and $F$ weakly preserves pullbacks of maps between
finite sets, then it weakly preserves all pullbacks.
\end{lem}

Weak preservation of pullbacks can be decomposed into two simpler
preservation conditions. We recall from \cite{GS05}:
\begin{lem}
A functor $F$ weakly preserves pullbacks iff $F$ weakly preserves
kernel pairs and preimages.
\end{lem}

In this note, we shall consider pullbacks of maps $\alpha,\beta$
where $im\,\alpha=im\,\beta$. Therefore, the following result is
relevant:
\begin{lem}
\label{lem:kernel pairs}For a $Set$-functor $F$ the following are
equivalent:
\begin{enumerate}
\item $F$ weakly preserves kernel pairs
\item $F$ weakly preserves pullbacks of epis
\item $F$ weakly preserves the pullback of maps $\alpha$ and $\beta$
for which $im\,\alpha=im\,\beta$.
\end{enumerate}
\end{lem}

The equivalence of $1.$ and $2.$ is due to the first author with
his student Ch.~Henkel, see \cite{henkel2010,Gumm20b}. The equivalence
of $2.$ and $3.$ is easily seen by epi-mono-factorization of $\alpha$
and of $\beta$. We obtain $\alpha=m\circ\alpha'$ and $\beta=m\circ\beta'$
where $m$ is mono and $\alpha'$ and $\beta'$ are epi. Then $Pb(\alpha,\beta)=Pb(\alpha',\beta')$.

\section{The free-algebra functor}

For any variety $\mathcal{V}$ of algebras of fixed signature $\tau,$
and for a set $X$ of variables, we denote by $F_{\mathcal{V}}(X)$
the $\mathcal{V}$-algebra freely generated by the set $X.$ Its defining
property is:
\begin{prop}
\emph{Given any algebra $A$ of type $\tau$ and given any set map
$\varphi:X\to A$, there is a unique homomorphism $\tilde{\varphi}:F_{\mathcal{V}}(X)\to A$
such that $\tilde{\varphi}\circ\iota_{X}=\varphi$ where $\iota_{X}$
denotes the inclusion of variables $X$ in $F_{\mathcal{V}}(X).$}
\[
\xymatrix{F_{\mathcal{V}}(X)\ar[r]^{\tilde{\varphi}} & A\\
X\ar@{^{(}->}[u]^{\iota_{X}}\ar[ur]_{\varphi}
}
\]
\end{prop}

In particular, starting with a map $f:X\to Y,$ and extending it with
$\iota_{Y}$ we obtain a map $\bar{f}:F_{\mathcal{V}}(X)\to F_{\mathcal{V}}(Y)$
as homomorphic extension of $\iota_{Y}\circ f.$

\[
\xymatrix{F_{\mathcal{V}}(X)\ar[r]^{\bar{f}} & F_{\mathcal{V}}(Y)\\
X\ar@{^{(}->}[u]^{\iota_{X}}\ar[r]^{f} & Y\ar@{^{(}->}[u]^{\iota_{Y}}
}
\]
It is easy to check that $\overline{id_{X}}=id_{F_{\mathcal{V}}(X)}$
and $\overline{g\circ f}=\bar{g}\circ\bar{f},$ hence $F_{\mathcal{V}}$
with object map $X\mapsto F_{\mathcal{V}}(X)$ and morphism map $f\mapsto\bar{f}$
is a functor. We shall consider $F_{\mathcal{V}}(X)$ as a set and
$\bar{f}$ as a set map when considering $F_{\mathcal{V}}$ as a set
functor. Thus, we suppress showing the application of the forgetful
functor U.

If $f$ in the above picture is surjective, then so is $\bar{f}.$
This is in fact so for any Set-functor $F.$ Namely, if $f:X\to Y$
is surjective it has a right-inverse $f^{-}$ such that $f\circ f^{-}=id_{Y},$
from which the functor properties yield $F(f)\circ F(f^{-})=id_{F(Y)},$
demonstrating that $F(f^{-})$ is a right inverse to $Ff$, which
therefore is surjective.

It is interesting to observe, even though it will not be needed for
the proof of our main result, that for a free-lattice functor $F_{\mathcal{L}}$,
with $\mathcal{L}$ a variety of lattices (without further operations),
the converse is almost true:
\begin{prop}
If $\mathcal{L}$ is a (quasi)-variety of lattices and $\varphi:F_{\mathcal{L}}(X)\twoheadrightarrow F_{\mathcal{L}}(Y)$
is a surjective homomorphism, then there is a subset $X_{0}\subseteqq X$
and a surjective map $f:X_{0}\twoheadrightarrow Y$ such that $\varphi$
restricted to $F_{\mathcal{L}}(X_{0})$ is $\bar{f}.$
\end{prop}

\[
\xymatrix{F_{\mathcal{L}}(X)\ar@{->>}[r]^{\varphi} & F_{\mathcal{L}}(Y)\\
F_{\mathcal{L}}(X_{0})\ar@{^{(}->}[u]^{\iota}\ar@{->>}[ur]_{\bar{f}}
}
\]

\begin{proof}
Each element $y\in Y$ must have a $\varphi-$preimage in $X,$ since
the free generators in any lattice free in $\mathcal{L}$ are both
$\vee$- and $\wedge$-irreducible, see \cite{FreeLattices}. Collecting
these preimages of $Y$ into a subset $X_{0}$ of $X,$ let $f$ be
the domain restriction of $\varphi$ to $X_{0}$. By construction,
$f:X_{0}\to Y$ is surjective, and $\varphi$ agrees with $\bar{f}$
on $F_{\mathcal{V}}(X_{0}).$
\end{proof}
Here we are interested in $F_{\mathcal{L}}$ where $\mathcal{L}$
is any (quasi)-variety of lattices, but we allow additional operations
in the signature, as long as the axioms of $\mathcal{L}$ force those
operations to be monotonic with respect to the lattice ordering. In
short: we assume that $\mathcal{L}$ is a quasi-variety of \emph{lattice-ordered
universal algebras}.

If $a,b\in A$ for such a lattice ordered algebra, we denote by $[a,b]$
the interval 
\[
[a,b]:=\{x\in A\mid a\le x\le b\},
\]
which is, of course, nonempty iff $a\le b.$

Let $X,Y,Z$ be finite sets. With $F_{\mathcal{L}}(X)$ we continue
to denote the free lattice-ordered algebra in $\mathcal{L}$ generated
by $X.$
\begin{lem}
\label{lem:interval_preimage}Let $g:X\twoheadrightarrow Y$ be a
surjective map and $\bar{g}:F_{\mathcal{L}}(X)\to F_{\mathcal{L}}(Y)$
the homomorphic extension of $g$. Then there are homomorphisms $\check{g},\hat{g}:F_{\mathcal{L}}(Y)\to F_{\mathcal{L}}(X)$
such that
\begin{enumerate}
\item $\bar{g}\circ\check{g}=id=\bar{g}\circ\hat{g}$,\label{enu:identity}
\item $(\hat{g}\circ\bar{g})p\le p\le(\check{g}\circ\bar{g})p$ ~for each
$p\in F_{\mathcal{L}}(X),$\label{enu:preimage_ordering}
\item for all $q_{1},q_{2}$ in $F_{\mathcal{L}}(Y)$ we have $\,\bar{g}^{-1}[q_{1},q_{2}]=[\hat{g}q_{1},\check{g}q_{2}],$
\item $\,\bar{g}^{-1}\{q\}=[\hat{g}q,\check{g}q]$~ for each $q\in F_{\mathcal{L}}(Y).$
\label{enu:preimage_is_interval}
\end{enumerate}
\end{lem}

\begin{proof}
Let $\hat{g}:F_{\mathcal{L}}(Y)\to F_{\mathcal{L}}(X)$ be the unique
homomorphism which for all $y\in Y$ is defined as 
\[
\hat{g}(y):=\bigwedge\{x\in X\mid gx=y\}
\]
and dually 
\[
\check{g}(y):=\bigvee\{x\in X\mid gx=y\}.
\]
\begin{enumerate}
\item Given $y\in Y$ then 
\begin{align*}
\bar{g}\hat{g}y & =\bar{g}(\bigwedge\{x\mid x\in X,gx=y\})\\
 & =\bigwedge\{\bar{g}x\mid x\in X,gx=y\}\\
 & =\bigwedge\{gx\mid x\in X,gx=y\}\\
 & =\bigwedge\{y\}\\
 & =y,
\end{align*}
 hence $\bar{g}\circ\hat{g}$ (and similarly $\bar{g}\circ\check{g})$
is the identity.
\item For each variable $x\in X$ we have $gx\in Y$, hence 
\[
(\hat{g}\circ\bar{g})x=\hat{g}(gx)=\bigwedge\{x'\in X\mid gx'=gx\}\le x.
\]
 For arbitrary terms $p=p(x_{1},\ldots,x_{n})\in F_{\mathcal{L}}(X)$
where $x_{i}\in X,$ we conclude
\begin{align*}
(\hat{g}\circ\bar{g})p & =(\hat{g}\circ\bar{g})p(x_{1},\ldots,x_{n})\\
 & =p(\hat{g}\bar{g}x_{1},\ldots,\hat{g}\bar{g}x_{n})\\
 & \le p(x_{1},\ldots,x_{n})\\
 & =p,
\end{align*}
and dually $(\check{g}\circ\bar{g})p\ge p.$
\item If $p\in\bar{g}^{-1}[q_{1},q_{2}]$ then $q_{1}\le\bar{g}(p)\le q_{2},$
so $\hat{g}q_{1}\le\hat{g}\bar{g}p\le p\le\check{g}\bar{g}p\le\check{g}q_{2}$,
by (\ref{enu:preimage_ordering}), hence $p\in[\hat{g}q_{1},\check{g}q_{2}].$
Conversely, if $p\in[\hat{g}q_{1},\check{g}q_{2}],$ then $\hat{g}q_{1}\le p\le\check{g}q_{2}$
so 
\[
q_{1}=\bar{g}\hat{g}q_{1}\le\bar{g}p\le\bar{g}\check{g}q_{2}=q_{2},
\]
 by (\ref{enu:identity}), so $p\in\bar{g}^{-1}[q_{1},q_{2}].$ Notice
that we do nowhere require $q_{1}\le q_{2}.$
\item This is a special case of the previous one, where $q_{1}=q=q_{2}.$
\end{enumerate}
Thus preimages of points are intervals. For our weak pullback preservation
property, we shall need joint preimages with respect to different
homomorphisms, which must be intersections of intervals.
\end{proof}
\begin{lem}
\label{lem:joint_interval_preimage}Given surjective maps $g_{1}:X\twoheadrightarrow Y$
and $g_{2}:X\twoheadrightarrow Z$ and given elements $p\in F_{\mathcal{L}}(Y),$
\textup{$q\in F_{\mathcal{L}}(Z)$ then the following equivalent conditions
state that $p$ and $q$ share a common preimage under $\bar{g_{1}}$
and $\bar{g_{2}}$ in $F_{\mathcal{L}}(X):$}
\begin{enumerate}
\item $\bar{g_{1}}^{-1}\{p\}\cap\bar{g_{2}}^{-1}\{q\}\ne\emptyset,$
\item $\hat{g}_{1}p\,\vee\,\hat{g}_{2}q\le\check{g}_{1}p\wedge\check{g}_{2}q$,
\item $\hat{g}_{1}p\le\check{g}_{2}q$ ~and ~ $\hat{g}_{2}q\le\check{g}_{1}p$
\item $\bar{g}_{1}\hat{g}_{2}q\le p$ and $\bar{g}_{2}\hat{g}_{1}p\le q.$
\end{enumerate}
\end{lem}

\begin{proof}
By the previous lemma,
\begin{align*}
g_{1}^{-1}\{p\}\cap g_{2}^{-1}\{q\}\ne\emptyset & \iff[\hat{g}_{1}p,\check{g}_{1}p]\cap[\hat{g}_{2}q,\check{g}_{2}q]\ne\emptyset\\
 & \iff[\hat{g}_{1}p\vee\hat{g}_{2}q,\check{g}_{1}p\wedge\check{g}_{2}q]\ne\emptyset\\
 & \iff\hat{g}_{1}p\vee\hat{g}_{2}q\le\check{g}_{1}p\wedge\check{g}_{2}q\\
 & \iff\hat{g}_{1}p\le\check{g}_{2}q\text{ and }\hat{g}_{2}q\le\check{g}_{1}p\\
 & \implies\bar{g}_{2}\hat{g}_{1}p\le\bar{g}_{2}\check{g}_{2}q=q\text{ and }\bar{g}_{1}\hat{g}_{2}q\le\bar{g}_{1}\check{g}_{1}p=p\\
 & \implies\hat{g}_{1}p\le\check{g}_{2}\bar{g}_{2}\hat{g}_{1}p\le\check{g}_{2}q\text{ and }\hat{g}_{2}q\le\check{g}_{1}\bar{g}_{1}\hat{g}_{2}q\le\check{g}_{1}p.
\end{align*}
\end{proof}
Combining the previous lemmas, we obtain:
\begin{lem}
\label{lem:solution_interval}Under the assumptions of Lemma \ref{lem:joint_interval_preimage},
if $p$ and $q$ have a common preimage under $g_{1}$ and $g_{2}$
then the set of all common preimages is the interval $[\hat{g}_{1}p\,\vee\,\hat{g}_{2}q,\check{g}_{1}p\wedge\check{g}_{2}q].$
\end{lem}

\section{Weak preservation of kernel pairs}

It is well known that the complete lattice of congruence relations
of any lattice, hence of any lattice ordered algebra, is distributive,
so in particular, $\mathcal{L}$ is \emph{congruence modular}. As
a corollary to a result from \cite{Gumm20b} it therefore follows,
that $F_{\mathcal{L}}$ will \emph{not} preserve preimages, hence
will \emph{not} weakly preserve all pullbacks. Fortunately, though,
this does not preclude $F_{\mathcal{L}}$ from preserving kernel pairs,
or equivalently, pullbacks of maps whose images agree. This is in
fact what we are proving now. Our main result is:
\begin{thm}
\label{thm:main}For any variety $\mathcal{L}$ of lattice-ordered
algebras the functor $F_{\mathcal{L}}$ weakly preserves pullbacks
of epis.
\end{thm}

From now on, whenever we denote terms $p,$ $q,$ $s$ as $p(x_{1},\ldots,x_{m})$,
$q(y_{1},\ldots,y_{n}),$ $s(z_{1},\ldots,z_{r})$ then we are implying
that their variables are mutually different, i.e. $x_{i}\not\ne x_{j}$,
$y_{i}\not\ne y_{j},$ $z_{i}\not\ne z_{j}$ unless $i=j.$ An equation
$p(u_{1},\ldots,u_{m})\approx q(v_{1},\ldots,v_{n}),$ arises from
substituting variables $u_{i},v_{j}$ for $x_{i}$ and $y_{j}$. For
that purpose we are allowed to have $u_{i}=u_{j}$ or $v_{i}=v_{j}$
even when $i\ne j.$ We denote the corresponding substitutions by
$u,$ resp. $v,$ hence 
\[
u_{i}=u(x_{i})\text{ and }v_{j}=v(y_{j}),
\]
so that $p(u_{1},\ldots,u_{m})=p(ux_{i},\ldots,ux_{m})=\bar{u}\,p(x_{1},\ldots,x_{m}),$
and $q(v_{1},\ldots,v_{m})=\bar{v}\,q(y_{1},\ldots,y_{n})$.

An equation $p(u_{1},\ldots,u_{m})\approx q(v_{1},\ldots,v_{n})$
is\emph{ }called \emph{balanced}, if the same variables occur on both
sides, i.e. $\{u_{1},\ldots,u_{m}\}=\{v_{1},\ldots,v_{n}\}$. With
these conventions and with the help of Prop.~\ref{prop:pointwise},
we can express Theorem.~\ref{thm:main} in purely universal algebraic
terms as follows:
\begin{thm}
\label{thm:Lattice}Let $\mathcal{L}$ be a (quasi-)variety of lattice
ordered algebras, $p(x_{1},\ldots,x_{m})$ and $q(y_{1},\ldots,y_{n})$
 terms and  
\begin{equation}
p(u_{1},\ldots,u_{m})\approx q(v_{1},\ldots,v_{n})\label{eq:pu_eq_qv}
\end{equation}
a balanced equation. Then there is a term $s(z_{1},\ldots,z_{k})$
with $k\le mn$, and variable substitutions $\sigma:\{z_{1},\ldots,z_{k}\}\to\{x_{1},\ldots,x_{m}\},$
and $\tau:\{z_{1},\ldots,z_{k}\}\to\{y_{1},\ldots,y_{n}\}$ so that
\begin{align}
p(x_{1},\ldots,x_{m}) & \approx s(\sigma z_{1},\ldots,\sigma z_{k})\label{eq:p_eq_s_sigma}\\
q(y_{1},\ldots,y_{m}) & \approx s(\tau z_{1},\ldots,\tau z_{k}).\label{eq:q_eq_s_tau}
\end{align}
and 
\begin{equation}
u\circ\sigma=v\circ\tau.\label{eq:alpha_sigma_eq_beta_tau}
\end{equation}
\end{thm}

The following figure illustrates the situation. Given terms $p(x_{1},\ldots,x_{m})$,
$q(y_{1},\ldots,y_{m})$ and a balanced equation $p(u_{1},\ldots,u_{m})\approx q(v_{1},\ldots,v_{n}),$
we find a common ancestor term $s(z_{1},\ldots,z_{r})$ so that both
$p(x_{1},\ldots,x_{m})$ and $q(y_{1},\ldots,y_{n})$ are instances
modulo the equations of $\mathcal{V},$ of $s$ by means of variable
substitutions $\sigma,$ resp. $\tau.$ Applying the substitutions
$u,$ resp. $v,$ which defined the original equation $p(u_{1},\ldots,u_{m})\approx q(v_{1},\ldots,v_{n})$,
we obtain $\gamma:=u\circ\sigma=v\circ\tau$, and thus a common substitution
instance of $s$ from which the original equation follows trivially
by $p(u_{1},\ldots,u_{m})\approx s(\gamma z_{1},\ldots,\gamma z_{r})\approx q(v_{1},\ldots,v_{n}):$

\[
\xyC{.3pc}\xymatrix{ &  & s(z_{1},\ldots,z_{r})\ar@{|->}[dr]^{\tau}\ar@{|->}[dl]_{\sigma}\ar@{|->}[dd]^{\gamma}\\
p(x_{1},\ldots,x_{m})\ar@{|->}[d]^{u}\ar@{}[r]|\approx & \overline{\sigma}\,s(z_{1},\ldots,z_{r})\ar@{|->}[d]^{u}\ar@{}[dr]^{\circ} &  & \overline{\tau}\,s(z_{1},\ldots,z_{r})\ar@{}[dl]_{\circ}\ar@{|->}[d]_{v}\ar@{}[r]|\approx & q(y_{1},\ldots,y_{n})\ar@{|->}[d]_{v}\\
p(u_{1},\ldots,u_{m})\ar@{}[r]|{\approx\,\,\,} & \bar{u}\,\bar{\sigma}\,s(z_{1},\ldots,z_{r})\ar@{}[r]|(0.52)= & \bar{\gamma}s(z_{1},\ldots,z_{r})\ar@{}[r]|(0.47)= & \bar{v}\,\bar{\tau}\,s(z_{1},\ldots,z_{r})\ar@{}[r]|{\,\,\,\,\approx} & q(v_{1},\ldots,v_{n})
}
\]
Formally: 
\begin{align*}
p(u_{1},\ldots,u_{m}) & =p(ux_{1},\ldots,ux_{m})\tag{def. of \ensuremath{u}}\\
 & =\bar{u}\,p(x_{1},\ldots,x_{m})\tag{by \ref{eq:homomorphic extension}}\\
 & \approx\bar{u}\,s(\sigma z_{1},\ldots,\sigma z_{k})\tag{by \ref{eq:p_eq_s_sigma}}\\
 & =\bar{u}\,\bar{\sigma}\,s(z_{1},\ldots,z_{k})\tag{by \ref{eq:homomorphic extension}}\\
 & =\overline{u\circ\sigma}\,s(z_{1},\ldots,z_{k})\tag{functor property}\\
 & =\overline{v\circ\tau}\,s(z_{1},\ldots,z_{k})\tag{by \ref{eq:alpha_sigma_eq_beta_tau}}\\
 & \approx\cdots\tag{ same arguments in reverse}\\
 & =q(v_{1},\ldots,v_{n}).
\end{align*}

We now come to the proof of Theorem \ref{thm:Lattice}.
\begin{proof}
Let $X:=\{x_{1},\ldots,x_{m}\}$, $Y:=\{y_{1},\ldots,y_{n}\}$, $U:=\{u_{1},\ldots,u_{m}\}$
and $V:=\{v_{1},\ldots,v_{n}\}$ be sets of variables with $U=V,$
$|X|=m$ and $|Y|=n$. Define $u(x_{i}):=u_{i}$ and $v(y_{i}):=v_{i}.$

Let $Pb(u,v)=\{(x,y)\in X\times Y\mid ux=vy\}$ be the pullback of
$u$ and $v$. The assumption $U=V=:W$ means that $\text{im}\,u$
= $\text{im}\,v,$ so 
\begin{equation}
\forall x\in X.\exists y\in Y.\,ux=vy,\label{eq:left_total}
\end{equation}
and symmetrically 
\begin{equation}
\forall y\in Y.\exists x\in X.\,ux=vy.\label{eq:right_total}
\end{equation}
These statements are equivalent to saying that the projections $\pi_{1}$
and $\pi_{2}$ from the pullback $Pb(u,v)$ to the components $X$
and $Y$ are surjective. 
\[
\xyC{1.5pc}\xymatrix{X\ar@{->>}[rr]^{u} &  & W\\
Pb(u,v)\ar@{->>}[rr]\sb(0.6){\pi_{2}}\ar@{->>}[u]^{\pi_{1}} & \, & Y\ar@{->>}[u]_{v}
}
\]
In applying the free-algebra functor $F_{\mathcal{L}}$ to this pullback-diagram,
we shall have to consider the elements of $Pb(u,v)$ as variables.
To emphasize this, we set $Z:=Pb(u,v)$ and write the elements of
$Z$ as follows: 
\begin{equation}
Z=\{z_{x,y}\mid ux=vy\}=\{z_{1},\ldots,z_{k}\},\label{eq:z_representation}
\end{equation}
so we retain the identities 
\begin{equation}
\pi_{1}z_{x,y}=x\label{eq:pi1zxy}
\end{equation}
and 
\begin{equation}
\pi_{2}z_{x,y}=y.\label{eq:pi2zxy}
\end{equation}

In order to show that $F_{\mathcal{L}}$ weakly preserves this pullback,
we must verify the conditions spelled out in Prop.~\ref{prop:pointwise}.
Thus given terms $p:=p(x_{1},\ldots,x_{m})\in F_{\mathcal{L}}(X)$
and $q:=q(y_{1},\ldots,y_{n})\in F_{\mathcal{L}}(Y)$ and an $\mathcal{L}-$equation
$p(u_{1},\ldots,u_{m})\approx q(v_{1},\ldots,v_{n}),$ we have $\bar{u}\,p(x_{1},\ldots,x_{m})=\bar{v}\,q(y_{1},\ldots,y_{n})=:r$
and must find some term $s(z_{1},\ldots,z_{k})$ such that 
\begin{equation}
\bar{\pi}_{1}s(z_{1},\ldots,z_{k})\approx p(x_{1},\ldots,x_{m})\label{eq:pi_s_eq_p}
\end{equation}
and likewise 
\begin{equation}
\bar{\pi}_{2}s(z_{1},\ldots,z_{k})\approx q(y_{1},\ldots,y_{n}).\label{eq:pi_s_eq_q}
\end{equation}
In other words, we are looking for a joint preimage $s(z_{1},\ldots,z_{k})$
of $p(x_{1},\ldots,x_{m})$ under $\bar{\pi}_{1}$ and of $q(y_{1},\ldots,y_{n})$
under $\bar{\pi}_{2}.$

This is where Lemma \ref{lem:joint_interval_preimage} comes into
play. We shall establish the last of its 4 equivalent conditions,
which means that we will prove that $\bar{\pi}_{1}\hat{\pi}_{2}q\le p$
and $\bar{\pi}_{2}\hat{\pi}_{1}p\le q.$ By symmetry, it suffices
to consider the first inequality, that is we need to check 
\[
\bar{\pi}_{1}\hat{\pi}_{2}q(y_{1},\ldots,y_{n})\le p(x_{1},\ldots,x_{m}).
\]
Hence the following lemma will complete the proof:
\end{proof}
\begin{lem}
\label{lem:Hilfslema}For $1\le i\le m$ and $1\le j\le n$ we have
\end{lem}

$q(\bar{\pi}_{1}\hat{\pi}_{2}y_{1},\ldots,\bar{\pi}_{1}\hat{\pi}_{2}y_{n})\le p(x_{1},\ldots,x_{m}).$
\begin{proof}
Let $g:=\bar{\pi}_{1}\circ\hat{\pi}_{2}$ then on variables $y\in Y:$
\begin{equation}
g(y)=\bar{\pi}_{1}\medwedge(z_{x,y}\mid u\,x=v\,y)=\medwedge(\bar{\pi}_{1}z_{x,y}\mid u\,x=v\,y)=\medwedge(x\mid u\,x=v\,y)\label{eq:g}
\end{equation}
Let $v^{-}$ be a right inverse to $v,$ which exists, as $v$ is
surjective. Observe that $\ker v\subseteq\ker g$, so by Lemma \ref{lem:diagram_lemma}
there exists a map $h:W\to F_{\mathcal{L}}(X)$ with $h\circ v=g.$
\[
\xymatrix{F_{\mathcal{L}}(X) & \ar@{_{(}->}[l]X\ar@{->>}[rr]^{u} & \, & W\ar@<-.2pc>[d]_{v^{-}}\ar@/_{2pc}/[lll]_{h}\\
 &  &  & Y\ar@<-.2pc>@{->>}[u]_{v}\ar@/^{1pc}/[ulll]^{g}
}
\]
It follows that 
\[
h=h\circ v\circ v^{-}=g\circ v^{-}.
\]
We now calculate: 
\begin{align*}
q(\bar{\pi}_{1}\hat{\pi}_{2}y_{1},\ldots,\bar{\pi}_{1}\hat{\pi}_{2}y_{n}) & =q(gy_{1},\ldots,gy_{n})\\
 & =q(hvy_{1},\ldots,hvy_{n})\\
 & =\tilde{h}\,q(vy_{1},\ldots,vy_{n})\\
 & =\tilde{h}\,q(v_{1},\ldots,v_{n})\\
 & \approx\tilde{h}\,p(u_{1},\ldots,u_{m})\\
 & =\tilde{h}\,p(ux_{1},\ldots,ux_{m})\\
 & =p(hux_{1},\ldots,hux_{m})\\
 & =p(gv^{-}ux_{1},\ldots,gv^{-}ux_{m})\\
 & \le p(x_{1},\ldots,x_{m}),
\end{align*}
where in the last step we invoked the observation that according to
(\ref{eq:g}):
\begin{align*}
gv^{-}ux_{i} & =\medwedge(x\mid ux=vv^{-}ux_{i})\\
 & =\medwedge(x\mid ux=ux_{i})\\
 & \le x_{i}
\end{align*}
together with the fact that all terms, in particular $p(x_{1},\ldots,x_{m}),$
are monotonic in each argument.
\end{proof}
So Theorem \ref{thm:Lattice} shows that one can always find an ancestor
term $s(z_{1},\ldots,z_{k})\in F_{\mathcal{L}}(Z)$ to $p$ and $q$
for any balanced equation $p(u_{1},\ldots,u_{m})=q(v_{1},\ldots,v_{n})$.
By Lemma \ref{lem:solution_interval} we conclude:
\begin{thm}
The set of all ancestor terms of $p(x_{1},\ldots,x_{m})$ and $q(y_{1},\ldots,y_{n})$
with respect to the balanced equation $p(u_{\ensuremath{1}},\ldots,u_{m})\approx q(v_{1},\ldots,v_{n})$
is the nonempty interval $[s_{0},s_{1}]$ in $F_{\mathcal{L}}(Z)$
whose bounds are given by 
\begin{equation}
s_{0}(z_{1},\ldots,z_{k})=p(\hat{\pi}_{1}x_{1},\ldots,\hat{\pi}_{1}x_{m})\vee q(\hat{\pi}_{2}y_{1},\ldots,\hat{\pi}_{2}y_{n})\label{eq:s0}
\end{equation}
and 
\begin{equation}
s_{1}(z_{1},\ldots,z_{k})=p(\check{\pi}_{1}x_{1},\ldots,\check{\pi}_{1}x_{m})\wedge q(\check{\pi}_{2}y_{1},\ldots,\check{\pi}{}_{2}y_{n}).\label{eq:s1}
\end{equation}
Here $\{z_{1},\ldots,z_{k}\}=\{z_{x_{i},y_{j}}\mid u_{i}=v_{j}\},$
with $\hat{\pi}_{1}x_{i}=\medwedge(z_{x_{i},y_{j}}\mid u_{i}=v_{j})$
and $\check{\pi}_{1}x_{i}=\bigvee(z_{x_{i},y_{j}}\mid u_{i}=v_{j}),$
and similarly $\hat{\pi}_{2}y_{j}=\medwedge(z_{x_{i},y_{j}}\mid u_{i}=v_{j})$
and $\check{\pi}_{2}y_{j}=\bigvee(z_{x_{i},y_{j}}\mid u_{i}=v_{j}).$
\end{thm}

As an exercise, the reader is invited to verify that the equation
\begin{equation}
p(x,x,y)\approx q(x,y,y)\label{eq:pxxy_eq_qxyy}
\end{equation}
discussed in the introductory Example \ref{exa:pxxy_eq_qxyy}, yields
the common ancestor term 
\[
s_{0}(z_{1},z_{2},z_{3},z_{4})=p(z_{1},z_{2},z_{3}\wedge z_{4})\vee q(z_{1}\wedge z_{2},z_{3},z_{4}).
\]
From this $p$ and $q$ can be obtained by identification of variables
\begin{align*}
p(x,y,z) & =s(x,y,z,z)\\
q(x,y,z) & =s(x,x,y,z)
\end{align*}
so that the original equation (\ref{eq:pxxy_eq_qxyy}) trivially results
from a further common identification: 
\[
p(x,x,y)=s(x,x,y,y)=q(x,y,y).
\]

\section{Extending the scope}

Looking beyond lattices and lattice ordered algebras, we find that
a theorem analogous to Theorem~\ref{thm:main} is also true for arbitrary
\emph{congruence permutable} varieties, often called \emph{Mal'cev
varieties}, such as groups, rings, quasigroups, etc.. These varieties
are also termed 2-\emph{permutable} in order to emphasize that they
belong to the more general class of \emph{$n$-permutable} varieties.
From \cite{Gumm20b} we quote:
\begin{prop}
\label{prop:Malcev}
\begin{enumerate}
\item \label{enu:2-perm_pres_ker_pairs}If $\mathcal{V}$ is a 2-permutable
variety then $F_{\mathcal{V}}$ weakly preserves kernel pairs
\item If $\mathcal{V}$ is $n$-permutable and $F_{\mathcal{V}}$ weakly
preserves kernel pairs, then $\mathcal{V}$ is 2-permutable.
\end{enumerate}
\end{prop}

This proposition also serves to document that there are indeed varieties
$\mathcal{V}$ for which $F_{\mathcal{V}}$ fails to preserve weak
pullbacks: The variety of \emph{implication algebras} is 3-permutable,
but not permutable, see \cite{Mitschke71}, hence:
\begin{cor}
If $\mathcal{V}$ is the variety of implication algebras, the free-algebra
functor $F_{\mathcal{V}}$ does not weakly preserve epi-pullbacks.
\end{cor}

Recall that by Mal'cev's theorem \cite{Malcev1954,Malcev1963} a variety
$\mathcal{V}$ is permutable iff there exists a ternary term $m(x,y,z)$
satisfying the equations (\ref{eq:malcev1}) and (\ref{eq:malcev2}).

To a permutable variety $\mathcal{V}$ we can by Prop.~\ref{enu:2-perm_pres_ker_pairs}
add arbitrary function symbols, yet $F_{\mathcal{V}}$ continues to
weakly preserve pullbacks. This behavior is different in the case
of lattice varieties, where we can only add monotonic operators, so
an extensional classification of all varieties $\mathcal{V}$ for
which $F_{\mathcal{V}}$ weakly preserves pullbacks may be difficult.

One might try to extend the scope from varieties and free-algebra
functors to a larger class of functors. One attempt would be to look
at monads, as is done in \cite{CHJ2013} and in \cite{Gumm20a}. Every
free-algebra functor is part of a monad $M=(F_{\mathcal{V}},\iota,\mu)$
where $\iota_{X}:X\to F_{\mathcal{V}}(X)$ and $\mu_{X}:F_{\mathcal{V}}(F_{\mathcal{V}}(X))\to F_{\mathcal{V}}(X)$
are obvious natural transformations. Hence in \cite{CHJ2013}, the
authors consider the question when monads are \emph{weakly cartesian},
meaning, that they weakly preserve pullbacks. However, as is noted
in \cite{Manes98,AlgebraicTheories}, finitary monads always arise
as above from algebraic theories, so only in the non-finitary case
a true generalization can be obtained.

A proper extension of our scope, however, is achieved by considering
$F_{\mathcal{V}}(-)$ as a \emph{copower functor}. Given an object
$A$ in a concrete category $\mathscr{C}$ (with forgetful functor
$U$) and a set $X$, let $A_{\mathscr{C}}[X]$ be the $X$-fold direct
sum in $\mathscr{C}$ of $A$ with itself, i.e. 
\[
A_{\mathscr{C}}[X]:=U(\coprod_{x\in X}A).
\]

It is easy to see that this construction is functorial. If $\mathcal{V}$
is a variety and $A\in\mathcal{V}$, then $\coprod_{x\in X}A$ exists
in $\mathcal{V}$, as was shown by Sikorski\cite{Sikorski52}, it
is in fact the same as the $X-$fold \emph{free product} of $A$ with
itself, see \cite{Graetzer}, pp 184 ff.. As a special case, the free
algebra with variables from $X$ is the $X-$fold sum in $\mathcal{V}$
of $F_{\mathcal{V}}(1)$, i.e. 
\[
F_{\mathcal{V}}(X)\cong\coprod_{x\in X}F_{\mathcal{V}}(1),
\]
so the free-algebra functor turns out to be a special instance of
a copower functor.

Monoids $\mathcal{M}$ for which the functor $\mathcal{M}_{\mathscr{C}}[-]$
weakly preserves preimages or pullbacks of epis have been characterized
with $\mathscr{C}$ being the variety $\mathfrak{M}$ of all monoids,
the variety $\mathfrak{Mc}$ of all commutative monoids or the variety
$\mathfrak{S}$ of all semigroups, see \cite{Gumm09}. The relevance
of $\mathcal{M}_{\mathfrak{Mc}}[-],$ for instance, arises from the
fact that one can argue that this functor models multisets (bags)
where the multiplicities of elements are counted by $\mathcal{M}.$

For lattices such an immediate Computer Science application is not
yet known, nevertheless would it be interesting to consider $L_{\mathcal{L}}[-]$
where $\mathcal{L}$ is the variety of lattices and $L$ an arbitrary
lattice. 

\section{Uniqueness and pullback preservation}

In category theoretical terms, uniqueness of the ancestor term would
amount to the free-algebra functor \emph{preserving }pullbacks of
epis (not just weakly). However, in \cite{CHJ2013}, the authors prove:
\begin{prop}
\label{prop:commutative_term}If $F_{\mathcal{V}}$ preserves pullbacks,
then every binary commutative term $t(x,y)$ is a pseudo-constant,
i.e. it satisfies $t(x,y)=t(z,z).$
\end{prop}

The term $x\wedge y$ therefore witnesses that for every nontrivial
variety $\mathcal{L}$ of lattices, the free-lattice functor $F_{\mathcal{L}}$
does not preserve pullbacks.

Below, we shall need a stronger version of this proposition which,
however, builds on the same proof idea. Given an equation $t(u_{1},\ldots,u_{n})\approx t(v_{1},\ldots,v_{n})$,
we shall reuse our notation from the proof of Theorem \ref{thm:Lattice}
and introduce new variables $z_{u_{i},u_{i}}$ as well as $z_{u_{i},v_{i}}$
for $1\le i\le n$.
\begin{prop}
\label{prop:independence}If $F_{\mathcal{V}}$ preserves pullbacks
of epis, then each term $t$ satisfying an equation $t(u_{1},\ldots,u_{n})\approx t(v_{1},\ldots,v_{n})$
also satisfies 
\[
t(z_{u_{1},u_{1}},\ldots,z_{u_{n},u_{n}})\approx t(z_{u_{1},v_{1}},\ldots,z_{u_{n},v_{n}}).
\]
\end{prop}

\begin{proof}
For $U=\{u_{1},\ldots,u_{n}\}$ and $V=\{v_{1},\ldots,v_{n}\}$ consider
the constant maps $\alpha:U\to\{x\}$ and $\beta:V\to\{x\},$ then
$Pb(\alpha,\beta)=U\times V$ and 
\[
\bar{\alpha}\,t(u_{1},\ldots,u_{n})=t(x,\ldots,x)=\bar{\beta}\,t(v_{1},\ldots,v_{n}).
\]
If $F_{\mathcal{V}}$ \emph{preserves} the pullback of $\alpha$ and
$\beta$, there ought be \emph{precisely one} term $s\in F_{\mathcal{V}}(Pb(\alpha,\beta))$
with $\bar{\pi}_{1}s=t(u_{1},\ldots,u_{n})$ and $\bar{\pi}_{2}s=t(v_{1},\ldots,v_{n}).$

However, we can present at least two candidates, namely

$s_{1}:=t((u_{1},u_{1}),\ldots,(u_{n},u_{n}))$ as well as $s_{2}:=t((u_{1},v_{1}),\ldots,(u_{n},v_{n}))$,
since 
\[
\bar{\pi}_{1}s_{1}\approx t(u_{1},\ldots,u_{n})\approx\bar{\pi}_{1}s_{2},
\]
and also 
\[
\bar{\pi}_{2}s_{1}\approx t(u_{1},\ldots,u_{n})\approx t(v_{1},\ldots,v_{n})\approx\bar{\pi}_{2}s_{2}.
\]
Hence 
\[
s_{1}=t((u_{1},u_{1}),\ldots,(u_{n},u_{n}))=t((u_{1},v_{1}),\ldots,(u_{n},v_{n}))=s_{2}.
\]
Recall that the elements $(u_{i},v_{j})\in Pb(f,g)$ act as variables
in $F_{\mathcal{V}}(Pb(f,g)),$ which we emphasize by writing $z_{u_{i},v_{j}}$
for the variable $(u_{i},v_{j})$ just like in the proof of Theorem
\ref{thm:Lattice}. Thus we infer the equation 
\[
t(z_{u_{1},u_{1}},\ldots,z_{u_{n},u_{n}})\approx t(z_{u_{1},v_{1}},\ldots,z_{u_{n},v_{n}}).
\]
\end{proof}
Full preservation of pullbacks seems to be an extremely strong condition
in the realm of free-algebra functors. We first demonstrate this for
permutable varieties. Given a Mal'cev term $m$ as in (\ref{eq:malcev1})
and (\ref{eq:malcev2}), then we can trivially infer the equation
\[
m(x,y,y)=m(y,y,x).
\]
Therefore, assuming that the free-algebra functor for a Mal'cev variety
$\mathcal{V}$ preserves pullbacks Prop. \ref{prop:independence}
yields the equation 
\[
m(z_{x,x},z_{y,y},z_{y,y})=m(z_{x,y},z_{y,y},z_{y,x}),
\]
which after renaming of variables can be written as 
\[
m(x,y,y)\approx m(u,y,v),
\]
thereby expressing the fact that $m$ must be independent of its first
and third index. With the help of either (\ref{eq:malcev1}) or (\ref{eq:malcev2}),
this implies $m(x,y,z)\approx m(y,y,y)\approx y,$ showing that $m$
is a projection operation, which then contradicts both (\ref{eq:malcev1})
and (\ref{eq:malcev2}). We conclude:
\begin{cor}
For each Mal'cev variety $\mathcal{V}$ the free-algebra functor $F_{\mathcal{V}}$
does not preserve epi-pullbacks, even though it does preserve them
weakly.
\end{cor}

Finally, we test Proposition \ref{prop:independence} on arbitrary
idempotent varieties. Recall that a variety $\mathcal{V}$ is called
\emph{idempotent}, when each fundamental operation $f$ satisfies
$f(x,\ldots,x)\approx x.$

It has been shown in \cite{Gumm20a} that for idempotent varieties
without constants the free-algebra functor $F_{\mathcal{V}}$ weakly
preserves products and pullbacks of constant maps. The following theorem
shows that nontrivial idempotent varieties will never (fully) preserve
pullbacks:
\begin{thm}
The only idempotent variety $\mathcal{V}$ for which $F_{\mathcal{V}}$
preserves pullbacks of epis, contains the ``variety of sets'' (where
all operations are implemented as projections).
\end{thm}

\begin{proof}
We employ a result of Olšák\cite{Olsak18} stating that any idempotent
variety satisfying at least one nontrivial equation (not satisfied
in the variety of sets) must have a six-ary term $t$ satisfying the
equations 
\[
t(x,y,y,y,x,x)\approx t(y,x,y,x,y,x)\approx t(y,y,x,x,x,y).
\]
Applying Prop.~\ref{prop:independence}, while renaming variables
$z_{x,x},z_{y,y},z_{x,y},z_{y,x}$ into $x,y,z,u$ we obtain from
the first equation the new 
\begin{equation}
t(x,y,y,y,x,x)\approx t(z,u,y,u,z,x).\label{eq:ol3}
\end{equation}
Hence also 
\[
t(y,y,x,x,x,y)\approx t(z,u,y,u,z,x),
\]
from which a second application of Prop. \ref{prop:independence}
with $z_{y,z},z_{y,u},z_{x,u},z_{x,z}$ renamed into variables $a,b,c,d$
yields the equation 
\begin{equation}
t(y,y,x,x,x,y)\approx t(a,b,z,c,d,u),\label{eq:o4}
\end{equation}
which clearly shows that $t$ is independent of any of its arguments,
so $t$ defines a pseudo constant. Hence by idempotency 
\[
x\approx t(x,x,x,x,x,x)\approx t(y,y,y,y,y,y)\approx y.
\]
\end{proof}

\section{Conclusion}

We have shown that every balanced equation $p(u_{1},\ldots,u_{m})=q(v_{1},\ldots,v_{n})$
in free lattice-ordered algebras can be derived from the fact that
$p$ and $q$ can be obtained by variable identification from a common
ancestor term $s,$ and the mentioned equation arises by further identifying
variables until a syntactically identical term is achieved. In category
theoretical language this means that the free algebra functor weakly
preserves pullbacks of epis.

\bibliographystyle{plain}
\bibliography{infobib}

\end{document}